\documentclass[12pt,a4paper]{amsart}

\usepackage[utf8]{inputenc}
\usepackage[T1]{fontenc}
\usepackage[english]{babel}

\usepackage{latexsym}
\usepackage{amsmath,amssymb,amsthm, amsfonts}
\usepackage{txfonts}						
\usepackage{stmaryrd}
\usepackage[normalem]{ulem}
\usepackage{cancel}
\usepackage{color}
\usepackage{remreset}      					
\usepackage[left=2cm, right=2cm, bottom=3cm]{geometry} 

\newtheorem{Theorem}{Theorem}
\newtheorem{Definition}[Theorem]{Definition}  
    
\newtheorem{Lemma}[Theorem]{Lemma}

\numberwithin{equation}{section} 

\newcommand{\R}{\mathbb{R}} 
\newcommand{\N}{\mathbb{N}} 

\newcommand{\abs}[1]{\left\vert #1 \right\vert}
\newcommand{\norm}[1]{\left\Vert #1 \right\Vert}
\newcommand{\paren}[1]{\left( #1 \right)}

\newcommand{\brac}[1]{\left\{ #1 \right\}}

\newcommand{\DEL}[1]{}

\allowdisplaybreaks

\makeatletter
\newcommand{\ps@bw}{\ps@empty%
  \renewcommand{\@oddfoot}%
  {\kuerzel \hfil {\footnotesize --- bitte wenden ---}}}
\newcommand{\ps@last}{\ps@empty%
  \renewcommand{\@oddfoot}{\kuerzel \hfil {\tiny\texturl}}}
\newcommand{\itemii}[1]{\refstepcounter{enumi}\hfill
  \hbox to 0.5\textwidth {\hskip \leftmargin \hskip -\labelwidth
    \hskip -\labelsep \hbox to\labelwidth {\makelabel{\@itemlabel}}%
    \hskip \labelsep #1\hfil}}
\newcommand{\itemiii}[1]{\refstepcounter{enumii}\hfill
  \hbox to 0.5\textwidth {\hskip \leftmargin \hskip -\labelwidth
    \hskip -\labelsep \hbox to\labelwidth {\makelabel{\@itemlabel}}%
    \hskip \labelsep #1\hfil}}
\makeatother

\DeclareMathOperator{\supp}{supp}
\DeclareMathOperator{\Div}{div}
\DeclareMathOperator{\curl}{curl}

\DeclareMathOperator{\esssup}{ess\ sup}
\DeclareMathOperator{\essinf}{ess\ inf}

\DeclareMathOperator{\rad}{rad}
\DeclareMathOperator{\cyl}{cyl}

\begin{document}

\makeatother
\makeatletter
\@removefromreset{Theorem}{section}
\makeatother


\title[Cylindrically symmetric ground states to a nonlinear curl-curl equation]{Existence of cylindrically symmetric ground states to a nonlinear curl-curl equation with non-constant coefficients}

\date{\today}

\author{Andreas Hirsch}
\address{A. Hirsch \hfill\break 
Institute for Analysis, Karlsruhe Institute of Technology (KIT), \hfill\break
D-76128 Karlsruhe, Germany}
\email{andreas.hirsch@kit.edu}

\author{Wolfgang Reichel}
\address{W. Reichel \hfill\break 
Institute for Analysis, Karlsruhe Institute of Technology (KIT), \hfill\break
D-76128 Karlsruhe, Germany}
\email{wolfgang.reichel@kit.edu}
  

  
\begin{abstract}
We consider the nonlinear curl-curl problem $\nabla\times\nabla\times U + V(x) U=f(x,\abs{U}^2)U$ in $\R^3$ related to the nonlinear Maxwell equations  with Kerr-type nonlinear material laws. We prove the existence of a symmetric ground-state type solution for a bounded, cylindrically symmetric coefficient $V$ and subcritical cylindrically symmetric nonlinearity $f$. The new existence result extends the class of problems for which ground-state type solutions are known. It is based on compactness properties of symmetric functions \cite{Lions2, Lions}, new rearrangement type inequalities from \cite{Brock} and the recent extension of the Nehari-manifold technique from \cite{SW}.  
\end{abstract}

\maketitle

\section{Introduction} 

We consider the system
\begin{align} \label{grund}
\nabla\times\nabla\times U + V(x)U = f(x,\abs{U}^2) U \text{ in } \R^3
\end{align}
where $V \in L^\infty(\R^3)$ and $f: \R^3\times [0,\infty)\to [0,\infty)$ is a non-negative Carath\'{e}odory function growing at infinity with a power at most $\frac{p-1}{2}$ for $p\in(1,5)$. The particular feature of \eqref{grund} is the curl-curl operator. It arises in specific models for standing waves in Maxwell's equations with Kerr-type nonlinear material laws where $f(x,|U|^2)U = \Gamma(x)|U|^2 U$. For a detailed physical motivation of (\ref{grund}) see \cite{BDPR}. 

\medskip

We look for $\R^3$-valued weak solutions $U$ in a cone $K_{4,1}$ of functions with suitable symmetries and $U\in L^2(\R^3) \cap L^{p+1}(\R^3)$, $\nabla\times U \in L^2(\R^3)$. The condition that $0$ lies below the spectrum of $\curl\curl+V(x)$ allows us to find ground-state type critical points of a functional $J(u) = \frac{1}{2} \|u\|^2-I(u)$, cf. \eqref{def_J}, restricted to the so-called Nehari-manifold. The basic idea of applying symmetrizations to minimizing sequences on the Nehari-manifold goes back to Stuart \cite{Stuart} in the context of the stationary nonlinear Schr\"odinger equation. Compared to \cite{Stuart} the assumptions on the nonlinearity $f$ can be substantially weakened beyond the classical Ambrosetti-Rabinowitz condition. This is based on three important ingredients: 
\begin{itemize}
\item the recent extension of the Nehari-manifold method due to Szulkin and Weth~\cite{SW},
\item the weak sequential continuity of functionals $I(u)$ and $I'(u)[u]$ on $K_{4,1}$ due to compactness properties of symmetric functions by Lions~\cite{Lions2, Lions},
\item new rearrangement inequalities for general nonlinearities due to Brock~\cite{Brock}.
\end{itemize}
Using the combination of these ingredients our main result of Theorem~\ref{gsthm} substantially extends the know results on the existence of ground-state type solutions for \eqref{grund}.

\medskip

Benci, Fortunato \cite{benci_fortunato_archive} and Azzollini, Benci, D'Aprile, Fortunato in \cite{Fortunato}
were among the first to consider the constant coefficient case of \eqref{grund} with $V\equiv 0$. Their method was based on cylindrical symmetries of the vector-fields $U$, cf. \cite{daprile_siciliano} for a different class of symmetries. The case where $f(x,\abs{U}^2) U= \Gamma(x) \abs{U}^{p-1}U$ with periodic coefficients $V$ and $\Gamma$ has been treated in \cite{BDPR}. 
In \cite{Mederski_2014} Mederski considered \eqref{grund} where $f(x,|U|^2)U$ is replaced by, e.g.,  $\Gamma(x) g(U)$ with
$\Gamma>0$ periodic and bounded, $V\leq 0$, $V\in L^\frac{p+1}{p-1}(\R^3)\cap L^\frac{q+1}{q-1}(\R^3)$ and $g(U)\sim |u|^{p-1} U$ if $|U|\gg 1$ and $g(U)\sim |U|^{q-1}u$ if $|U|\ll 1$ for $1<p<5<q$. A remarkable feature of Mederski's work is that \eqref{grund} can be treated without assuming special symmetries of the field $U$. The nonlinear curl-curl problem on bounded domains with the boundary condition $\nu\times U=0$ has been elaborated in \cite{bartsch_mederski, bartsch_mederski_2}.

\medskip

An important feature of \cite{Fortunato} is the use of cylindrically symmetric ansatz functions for $U$. Here we make a slightly different ansatz of the form
\begin{align} \label{richtigeransatz}
U(x)=u(r,z) \begin{pmatrix}-x_2\\x_1\\0\end{pmatrix} \text{ where } r=\sqrt{x_1^2+x_2^2},\  z=x_3.
\end{align}
Moreover, we assume cylindrically symmetric coefficients $V(x)=V(r,z),f(x,\abs{U}^2)=f(r,z,\abs{U}^2)$. For $U$ of the form (\ref{richtigeransatz}) we see that $\Div U=0$, and hence (\ref{grund}) reduces to the scalar equation
\begin{align} \label{scaequ}
-\frac{1}{r^3}\frac{\partial}{\partial r}\paren{r^3 \frac{\partial u}{\partial r}} -\frac{\partial^2 u}{\partial z^2} + V(r,z) u = f(r,z,r^2u^2)u \quad \mbox{ for }\quad (r,z)\in\Omega\coloneqq (0,\infty)\times\R.
\end{align}
It turns out that a suitable space to consider \eqref{scaequ} is given by 
\begin{align*}
H^1_{\cyl}(r^3drdz)\coloneqq & \brac{v\colon (0,\infty)\times\R\to\R: v, \frac{\partial v}{\partial r}, \frac{\partial v}{\partial z}\in L^2_{\cyl}(r^3drdz)},\\
L^2_{\cyl}(r^3drdz)\coloneqq & \brac{v\colon (0,\infty)\times\R\to\R: \int_\Omega v(r,z)^2 r^3d(r,z)<\infty},
\end{align*}
cf. Section~\ref{2} for more details on these spaces. Weak solutions of (\ref{scaequ}) arise as critical points of the functional
\begin{align} \label{def_J}
J(u)=\frac{1}{2}\int_{\Omega} \paren{\vert\nabla_{r,z} u\vert^2+V(r,z) u^2} r^3 d(r,z) -\int_{\Omega} \frac{1}{2r^2} F(r,z,r^2u^2) r^3 d(r,z), \ \ u\in H^1_{\cyl}(r^3 drdz),
\end{align}
where $F(r,z,t)\coloneqq \int_0^t f(r,z,s)\,ds$ and $\nabla_{r,z} \coloneqq \paren{\frac{\partial}{\partial r}, \frac{\partial}{\partial z}}$. A ground state $u$ of (\ref{scaequ}) is defined as a weak solution of (\ref{scaequ}) in the Nehari-manifold
\begin{align*}
M\coloneqq \brac{v\in H^1_{\cyl}(r^3 drdz)\setminus \{0\}: \int_\Omega \paren{\vert\nabla_{r,z} v\vert^2+V(r,z)v^2} r^3d(r,z) = \int_\Omega f(r,z,r^2v^2)v^2r^3 d(r,z)}
\end{align*}
such that
\begin{align*}
J(u)=\inf_{v\in M} J(v),
\end{align*}
see the classical papers \cite{Nehari1}, \cite{Nehari2}. We find ground states of (\ref{scaequ}) under additional assumptions on $V$ and $f$. To state these assumptions we need the notion of Steiner-symmetrization, cf. Chapter 3 in \cite{LiebLoss}. The Steiner-symmetrization (also called symmetric-deacreasing rearrangement) of a cylindrical function $g=g(r,z)$ with respect to $z$ is denoted by $g^\star$. We say that $g$ is \textit{Steiner-symmetric} if $g$ coincides with its Steiner-symmetrization with respect to $z$, keeping the $r$-variable fixed. A function $h\in L^\infty(\Omega)$ is \textit{ reversed Steiner-symmetric} if $\paren{\esssup h-h}^\star=\esssup h-h$ holds true.

\medskip

Now we can state our assumptions on $f$.
\begin{itemize}
\item[(i)] $f: \Omega\times [0,\infty)\to \R$ is a Carath\'{e}odory function with $0 \leq f(r,z,s)\leq c(1+s^\frac{p-1}{2})$ for some $c>0$ and $p\in (1,5)$,
\item[(ii)] $f(r,z,s)=o(1)$ as $s\to 0$ uniformly in $r, z\in [0,\infty)\times \R$,
\item[(iii)] $f(r,z,s)$ strictly increasing in $s\in [0,\infty)$,
\item[(iv)] $\frac{F(r,z,s)}{s}\to \infty$ as $s\to\infty$ uniformly in $r, z\in [0,\infty)\times \R$,
\item[(v)] for all $r \in [0,\infty)$, $s\geq 0$ and $\sigma>0$ the function 
$$
\varphi_\sigma(r,z,s)\coloneqq f(r,z,(s+\sigma)^2)(s+\sigma)^2-f(r,z,s^2)s^2
$$ 
is symmetrically nonincreasing in $z$.
\end{itemize}
Conditions (ii)--(iv) are inspired by the work of Szulkin and Weth \cite{SW}. Namely, if we translate (ii)--(iv) into conditions for $\tilde f(r,z,s):= f(r,z,r^2 s^2)s$ then they become identical to (ii)--(iv) of Theorem 20 from \cite{SW}. Condition (v) is used to prove the rearrangement inequality of Lemma~\ref{rearrang_inequ} and it is due to Brock \cite{Brock}.

\medskip

Next we state our main result.
\begin{Theorem} \label{gsthm}
Let $V\in L^\infty (\Omega)$ be reversed Steiner-symmetric such that the map 
\begin{align} \label{aequivnorm}
\norm{\cdot}\colon H^1_{\cyl}(r^3drdz)\to \R; u \mapsto \left(\int_\Omega \paren{\vert\nabla_{r, z}u\vert^2+ V(r,z)u^2} r^3 d(r,z)\right)^\frac{1}{2}
\end{align}
is an equivalent norm to $\norm{\cdot}_{H^1_{\cyl}(r^3dr dz)}$.
Additionally, let $f$ satsify the assumptions (i)-(v). Then (\ref{scaequ}) has a ground state $u\in H^1_{\cyl}(r^3drdz)$ which is symmetric about $\{z=0\}$.
\end{Theorem}

\noindent 
{\bf Remarks:} (1) The assumption of norm-equivalence is for instance satisfied if $V\geq 0$ and $\inf_{B_R^c} V>0$ for some $R>0$, where $B_R^c\coloneqq \{(r,z)\in\Omega: r^2+z^2>R^2\}$. For the reader's convenience the proof based on Poincaré's inequality is given in the Appendix. Since Poincaré's inequality is applicable for domains bounded in one direction we can weaken $\inf_{B_R^c} V>0$ to $\inf_{S^c}V>0$ for strips $S=[0,\infty) \times [0,\rho]$ with $\rho>0$ or $S=[r_0,r_1] \times [0,\infty)$ with $0\leq r_0<r_1<\infty$.

(2) The conditions on $f$ are satisfied if for instance $f(r,z,s)=\Gamma(r,z) |s|^\frac{p-2}{2}s$ where $\Gamma\in L^\infty(\Omega)$ is Steiner-symmetric, $\essinf_\Omega \Gamma>0$ and $p\in (1,5)$. This choice of $f$ corresponds to the equation $\nabla\times\nabla\times U +V(r,z) U=\Gamma(r,z)\abs{U}^{p-1}U$ in $\R^3$. Another possible choice is $f(r,z,s)=\Gamma(r,z) \log (1+s)$ where again $\Gamma\in L^\infty(\Omega)$ is Steiner-symmetric and $\essinf_\Omega \Gamma>0$. This nonlinearity appeared for instance in \cite{Liu} and it does not satisfy the classical Ambrosetti-Rabinowitz condition.
\medskip

The paper is structured as follows: In Section~\ref{2} we give details on the variational formulation of problem (\ref{scaequ}) and prove pointwise decay estimates of Steiner-symmetric functions in $H^1_{\cyl}(r^3drdz)$. In Section~\ref{3} we give the proof of Theorem~\ref{gsthm}, and in the Appendix we show an example for the potential $V$ satisfying the equivalent-norm-assumption of Theorem~\ref{gsthm}.

\section{Variational formulation, decay estimates, rearrangements} \label{2}

Let us consider some properties of the space $H^1_{\cyl}(r^3drdz)$. First, for $U$ of the form \eqref{richtigeransatz} we have that $U\in H^1(\R^3)$ if and only if $u\in H^1_{\cyl}(r^3drdz)$. A norm on $H^1_{\cyl}(r^3drdz)$ is given by
\begin{align*}
\norm{u}_{H^1_{\cyl}(r^3drdz)}\coloneqq \left(\int_\Omega \paren{\vert\nabla_{r,z} u\vert^2 + u^2}r^3 d(r,z)\right)^\frac{1}{2}.
\end{align*}
Notice that the space $H^1_{\cyl}(r^3drdz)$ behaves like a Sobolev-space in dimension $5$. Next we show a useful embedding property. For this we need the following Sobolev and Lebesgue spaces in dimension $3$ together with their canonical norms:
\begin{align*}
H^1_{\cyl}(rdrdz)\coloneqq & \brac{v\colon (0,\infty)\times\R\to\R: v, \frac{\partial v}{\partial r}, \frac{\partial v}{\partial z}\in L^2_{\cyl}(rdrdz)},\\
L^q_{\cyl}(rdrdz)\coloneqq & \brac{v\colon (0,\infty)\times\R\to\R: \int_\Omega |v(r,z)|^q rd(r,z)<\infty} \text{ for } q\in[1,\infty).
\end{align*}
\begin{Lemma} For $u\in H^1_{\cyl}(r^3drdz)$ Hardy's inequality holds
\begin{align} \label{HardyU}
\int_\Omega \frac{u^2}{r^2} r^3d(r,z) \leq C_H\int_\Omega \paren{\paren{\frac{\partial u}{\partial r}}^2+\paren{\frac{\partial u}{\partial z}}^2}r^3 d(r,z).
\end{align}
Moreover, if $u \in H^1_{\cyl}(r^3drdz)$ then $ ru \in H^1_{\cyl}(rdrdz)$ and there is a constant $C>0$ such that for $2\leq q \leq 6$
\begin{equation} \label{hilfreiche_ungl}
\norm{ru}_{H^1_{\cyl}(r dr dz)}, \|ru\|_{L^q_{\cyl}(rdrdz)} \leq C \norm{u}_{H^1_{\cyl}(r^3dr dz)}
\end{equation}
\end{Lemma}

\begin{proof}
Hardy's inequality \eqref{HardyU} is given in Lemma 9 (i) in \cite{BDPR}. For $u\in H^1_{\cyl}(r^3drdz)$ we have $ru$, $\frac{\partial}{\partial z}\paren{ru}$, $r\frac{\partial u}{\partial r} \in L^2_{\cyl}(rdr dz)$ and by \eqref{HardyU} also $u\in L^2_{\cyl}(r dr dz)$. Since $\frac{\partial}{\partial r}\paren{ru}=r\frac{\partial u}{\partial r}+u$ we conclude altogether $ru\in H^1_{\cyl}(rdrdz)$. By the Sobolev embedding in three dimensions this implies $ru\in L^q(r dr dz)$ for $q\in [2,6]$ and (\ref{HardyU}) yields
\begin{align} \label{H^1(r)gegenH^1(r^3)}
\begin{split}
\norm{ru}^2_{H^1_{\cyl}(r dr dz)} &= \int_\Omega \paren{\vert\nabla_{r,z}(ru)\vert^2 + r^2 u^2} rd(r,z) \\
&\leq 2 \int_\Omega \paren{\paren{r\frac{\partial u}{\partial z}}^2+ \paren{r\frac{\partial u}{\partial r}}^2 + u^2 + r^2u^2} r d(r,z)\leq \tilde{C}\norm{u}^2_{H^1_{\cyl}(r^3 dr dz)}.
\end{split}
\end{align}
\end{proof}

Next we show that the functional $J$ from the introduction as well as the functional in the defintion of the Nehari-manifold are well-defined. 

\begin{Lemma} \label{Lemmawohlgestellt}
There is a constant $C>0$ such that
\begin{align} \label{Einbettung}
\int_\Omega f(r,z,r^2 u^2)u^2 r^3\,drdz, \int_{\Omega} \frac{1}{2r^2} F(r,z,r^2u^2) r^3 d(r,z) \leq C \paren{\norm{u}_{H^1_{\cyl}(r^3dr dz)}^2+\norm{u}_{H^1_{\cyl}(r^3dr dz)}^{p+1}} 
\end{align}
for all $u\in H^1_{\cyl}(r^3dr dz)$.
\end{Lemma}

\begin{proof}
Clearly assumption (i) and (ii) show that for every $\epsilon>0$ there is $C_\epsilon>0$ such that 
\begin{align*}
0\leq f(r,z,s) \leq \epsilon +C_\epsilon s^\frac{p-1}{2}.
\end{align*} 
Hence 
\begin{align} 
0\leq f(r,z,r^2 u^2)u^2r^3 & \leq \left(\epsilon r^2u^2 + C_\epsilon \vert ru\vert^{p+1})\right) r, \label{estimate_f}\\
0\leq \frac{1}{2r^2}F(r,z,r^2u^2)r^3  & \leq \left(\epsilon r^2u^2 + \tilde C_\epsilon \vert ru\vert^{p+1}\right)r. \label{estimate_F}
\end{align}
Due to \eqref{hilfreiche_ungl} this implies the claim. 
\end{proof}

In order to find critical points of $J$ we need uniform decay estimates of Steiner-symmetric functions in $H_{\cyl}^1(r^3drdz)$. These estimates are given in \cite{Lions} in much more generality but for the sake of completeness we give them here together with the simple proof. We start with a well-known fact concerning radially symmetric functions and afterwards extend the result to cylindrically symmetric functions. Let 
\begin{align*}
H_{\rad}^1 (\R^n)\coloneqq \brac{u\in H^1(\R^n): u \text{ is radially symmetric}}.
\end{align*} 

\begin{Lemma} (see \cite{Lions}) \label{radabkling}
Let $n\geq 2$. Then there is a constant $C>0$ such that
\begin{align*}
\abs{u(x)}\leq C \norm{\nabla u}_{L^2(\R^n)}^{1/2}\norm{u}_{L^2(\R^n)}^{1/2} \abs{x}^{-(n-1)/2} \text{ for almost all } x\in\R^n \text{ and all } u\in H_{\rad}^1(\R^n).
\end{align*}
\end{Lemma}
\begin{proof}
By density it is sufficient to prove the estimate for $u\in H_{\rad}^1(\R^n)\cap C_c^\infty (\R^n)$. Let $r\coloneqq \abs{x}$. Then
\begin{align*}
\frac{d}{dr}\paren{r^{n-1}\abs{u}^2} = (n-1) r^{n-2}\abs{u}^2 + r^{n-1} 2 u \frac{\partial u}{\partial r} \geq -2 \abs{u} \abs{\frac{\partial u}{\partial r}} r^{n-1}.
\end{align*}
Integrating from $r$ to $\infty$ and expanding the domain of integration to all of $\R^n$ yields
\begin{align*}
r^{n-1} \abs{u(x)}^2 \leq C\int_{\R^n} \abs{u}\abs{\nabla u} dy \leq C \norm{\nabla u}_{L^2(\R^n)} \norm{u}_{L^2(\R^n)}. \tag*{\qedhere} 
\end{align*}
\end{proof}

Now we give an extension of Lemma~\ref{radabkling} to cylindrically symmetric functions which are Steiner-symmetric in the non-radial component. We make use of the following notation: Let $t\in \N_{\geq 2}$ and $s\in \N$ such that $n=t+s$. We write points in $\R^n$ as $(x,y)$ with $x\in\R^t$ and $y=\paren{y_1,\dots, y_s}\in \R^s$. Furthermore, let 
\begin{align*}
K_{t,s}\coloneqq \brac{u\in H^1(\R^n) \text{\ s.t. } \begin{cases} u(\cdot, y) &\text{ is a radially symmetric function for every } y\in\R^s \text{ and } \\ u(x,\cdot) &\text{ is Steiner-symmetric w.r.t. } y_i ,i=1,\dots,s, \text{ for every } x\in\R^t \end{cases}}.
\end{align*}
In particular, if $u\in K_{t,s}$ then necessarily $u\geq 0$. In this setting we have the following extension of Lemma \ref{radabkling}.

\begin{Lemma} (see \cite{Lions}) \label{zylabkling}
There is a constant $C>0$ such that
\begin{align*}
0\leq u(x,y)\leq C \norm{\nabla_x u}^{1/2}_{L^2 (\R^n)}\norm{u}^{1/2}_{L^2 (\R^n)} \abs{x}^{-(t-1)/2} \abs{y_1\cdots y_s}^{-1/2} \text{ for almost all } (x,y)\in\R^n \text{ and all } u\in K_{t,s}.
\end{align*}
\end{Lemma}

\begin{proof}
Let $u\in K_{t,s}$ and fix $y\in \R^s$. W.l.o.g. let $y_i>0$ for all $i=1,\dots,s$. We define
\begin{align*}
v(x)\coloneqq \int_0^{y_1} \cdots \int_0^{y_s} u(x,z) dz \text{ for } x\in\R^t.
\end{align*}
By Hölder's inequality we obtain $v^2(x)\leq y_1 \cdots y_s \int_0^{y_1} \cdots \int_0^{y_s} u^2(x,z) dz$, i.e.,
\begin{align} \label{lalpha}
\norm{v}_{L^2 (\R^t)}&\leq (y_1 \cdots y_s)^{1/2} \norm{u}_{L^2 (\R^n)}.
\end{align}
In the same manner we receive
\begin{align} \label{lbeta}
\norm{\nabla v}_{L^2 (\R^t)} \leq (y_1\cdots y_s)^{1/2} \norm{\nabla_x u}_{L^2 (\R^n)}.
\end{align}
Since $v\colon \R^t\to\R$ is radially symmetric we can apply Lemma \ref{radabkling} and get from (\ref{lalpha}) and (\ref{lbeta})
\begin{align} \label{zylabk2}
0\leq v(x) &\leq C \norm{\nabla v}_{L^2 (\R^t)}^{1/2} \norm{v}^{1/2}_{L^2 (\R^t)} \abs{x}^{-(t-1)/2} \leq C (y_1\cdots y_s)^{1/2} \norm{\nabla_x u}^{1/2}_{L^2 (\R^n)} \norm{u}^{1/2}_{L^2 (\R^n)} \abs{x}^{-(t-1)/2}. 
\end{align}
Due to the monotonicity-property in $y$-direction we also have $v(x)\geq y_1\cdots y_s u(x,y)$ and thus (\ref{zylabk2}) gives the desired inequality.
\end{proof}

We prove three additional lemmas which are used in the next section.
\begin{Lemma} \label{Kabg}
The set $K_{t,s}$ is a weakly closed cone in $H^1(\R^n)$.
\end{Lemma}

\begin{proof}
Take a sequence $\paren{u_k}_{k\in\N}\subset K_{t,s}$ such that $u_k \rightharpoonup u\in H^1(\R^n)$ as $k\to\infty$. By the Sobolev embedding on bounded domains we deduce that a subsequence of $u_k$ converges pointwise almost everywhere on $\R^n$ to $u$. Since every $u_k$ enjoys the radial symmetry in the first component and the non-increasing property in the second variable, the pointwise convergence implies that also $u$ enjoys these properties, i.e., $u\in K_{t,s}$.   
\end{proof}

\begin{Lemma}\label{weak_continuity}
The functionals 
$$
I(v)= \int_\Omega \frac{1}{2r^2} F(r,z,r^2 v^2)r^3\,d(r,z), \qquad I'(v)[v] = \int_\Omega f(r,z,r^2v^2)v^2 r^3\,d(r,z)
$$
are weakly sequentially continuous on the set $K_{4,1}\subset H^1_{\cyl}(r^3drdz)$.
\end{Lemma}

\noindent
{\bf Remark:} In the proof we use twice the following principle: if $S\subset \R^m$ is a set of finite measure and $w_k: S\to\R$ a sequence of measurable functions such that $\|w_k\|_{L^r(S)}\leq C$ and $w_k\to w$ pointwise a.e. as $k\to \infty$ then $\|w_k-w\|_{L^q(S)} \to 0$ as $k\to \infty$ for $1\leq q<r$. The proof is as follows: Egorov's theorem allows to choose $\Sigma\subset S$ such that $w_k\to w$ uniformly on $\Sigma$ and $|S\setminus\Sigma|\leq \epsilon$ arbitrary small. By H\"older's inequality the remaining integral is estimated by $\int_{S\setminus\Sigma} |w_k-w|^q\,dx \leq \epsilon^{1-\frac{q}{r}} \|w_k-w\|_{L^r(S)}^q$.

\begin{proof} Let us take a weakly convergent sequence $(v_k)_{k\in \N}$ in $K_{4,1}$ such that $v_k \rightharpoonup v$ in $H^1_{\cyl}(r^3drdz)$ and $v_k\to v$ pointwise a.e. in $\Omega$. By Lemma~\ref{Kabg} one gets $v\in K_{4,1}$ and using Lemma~\ref{zylabkling} there exists a constant $C>0$ such that 
\begin{equation}
\label{abkling_5d}
0 \leq v_k(r,z), v(r,z) \leq C r^{-\frac{3}{2}} |z|^{-\frac{1}{2}} \mbox{ for all } k\in \N \mbox{ and almost all } (r,z) \in \Omega.
\end{equation}
Our goal is now to show at least for a subsequence
\begin{equation} \label{zielvvinfty}
\int_\Omega \frac{1}{r^2}F(r,z,r^2v_k^2) r^3d(r,z) \to \int_\Omega \frac{1}{r^2}F(r,z,r^2v^2) r^3d(r,z) \text{ as } k\to\infty 
\end{equation}
and 
\begin{equation}\label{zielvvinfty_2}
\int_\Omega f(r,z,r^2v_k^2)v_k^2 r^3d(r,z) \to \int_\Omega f(r,z,r^2v^2)v^2 r^3d(r,z) \text{ as } k\to\infty.
\end{equation}
By \eqref{estimate_F} we find 
\begin{align*}
\frac{1}{r^2}\left|F(r,z,r^2 v_k^2)-F(r,z,r^2 v^2) \right|r^3 \leq \epsilon r^2(v_k^2+v^2) r + C_\epsilon \paren{\vert rv_k\vert^{p+1}+\vert rv\vert^{p+1}} r
\end{align*}
and hence 
\begin{equation}
\label{key_estimate_2}
 \left(|F(r,z,r^2 v_k^2)-F(r,z,r^2 v^2)|- \epsilon r^2(v_k^2+v^2)\right)^+ r \leq  C_\epsilon \paren{\vert rv_k\vert^{p+1}+\vert rv\vert^{p+1}} r.
\end{equation}
Inspired by \cite{Lions2} and \cite{Lions} the idea is to show 
\begin{align} \label{Inp+1cpt}
rv_k \to rv \text{ in } L^{p+1}(rdrdz) \text{ as } k\to\infty.
\end{align}
Once \eqref{Inp+1cpt} is established we obtain a majorant  $|rv_k|, |rv|\leq w\in L^{p+1}(r\,drdz)$ (cf. Lemma A.1 in \cite{Willem}). Together with \eqref{key_estimate_2} this majorant allows to apply Lebesgue's dominated convergence theorem and yields 
\begin{equation} \label{new_estimate}
\lim_{k\to \infty} \int_\Omega \left(|F(r,z,r^2 v_k^2)-F(r,z,r^2 v^2)|- \epsilon r^2(v_k^2+v^2)\right)^+ r\,drdz = 2\epsilon \|v\|^2_{L^2(r^3drdz)}. 
\end{equation}
If we set
$$
a_k := \int_\Omega |F(r,z,r^2 v_k^2)-F(r,z,r^2 v^2)| r\,drdz 
$$
and
$$
b_k := \epsilon\|r^2(v_k^2+v^2)\|_{L^1(rdrdz)} =\epsilon(\|v_k\|_{L^2(r^3drdz)}^2+ \|v\|_{L^2(r^3drdz)}^2) \leq C\epsilon
$$
then 
\begin{align*}
\limsup_{k\in\N} a_k & \leq \limsup_{k\in \N} b_k + \limsup_{k\in \N} (a_k-b_k)^+ \\
& \leq C\epsilon + \limsup_{k\in \N} \left( \int_\Omega  \left(|F(r,z,r^2 v_k^2)-F(r,z,r^2 v^2)|- \epsilon r^2(v_k^2+v^2)\right)rdrdz\right)^+ \\
& \leq C\epsilon + \limsup_{k\in \N} \int_\Omega \left(|F(r,z,r^2 v_k^2)-F(r,z,r^2 v^2)|- \epsilon r^2(v_k^2+v^2)\right)^+rdrdz \\
& \leq \epsilon(C+2\|v\|^2_{L^2(r^3drdz)}) \mbox{ by \eqref{new_estimate}}.
\end{align*}
Since $\epsilon>0$ was arbitrary this shows that $\lim_{k\to\infty} a_k=0$ and therefore \eqref{zielvvinfty} holds. The proof of \eqref{zielvvinfty_2} is similar since $\left(f(r,z,r^2v_k^2)r^2v_k^2-f(r,z,r^2 v^2)r^2v^2-\epsilon r^2(v_k^2+v^2)\right)^+r$ satisfies an estimate just like \eqref{key_estimate_2} if we use \eqref{estimate_f} instead of \eqref{estimate_F}.

\smallskip

It remains to prove \eqref{Inp+1cpt}. For this, we split our domain $\Omega$ into four parts $\Omega_1,\dots, \Omega_4$ and show \eqref{Inp+1cpt} on each of these parts separately. The definitions of $\Omega_1,\dots, \Omega_4$ are as follows: For $R>0$ let
\begin{align*}
\Omega_1\coloneqq  \{(r,z)\in\Omega: r< R, \abs{z}< R\}, \ \ \ 
&\Omega_2\coloneqq  \{(r,z)\in\Omega: r\geq R, \abs{z}\geq R\}, \\
\Omega_3\coloneqq  \{(r,z)\in\Omega: r< R, \abs{z}\geq R\}, \ \ \ 
&\Omega_4\coloneqq  \{(r,z)\in\Omega: r\geq R, \abs{z}< R\}.
\end{align*}  

\smallskip

Convergence on $\Omega_1$: Follows from $rv_k \to rv$ in $L^q(K; r\,drdz)$ for every compact subset $K\subset [0,\infty)\times \R$ and every $q \in [1,6)$. This step works independently of the choice of $R>0$.

\smallskip

Convergence on $\Omega_2$: Let $\varepsilon>0$. With the help of \eqref{abkling_5d} we calculate
\begin{align*}
\int_{\Omega_2} \abs{rv_k-rv}^{p+1} r d(r,z) &\leq 2^{p+1} \int_{\Omega_2} r^{p+1} \paren{\vert v_k\vert^{p+1}+\vert v\vert^{p+1}} r d(r,z)   \\ 
&\leq 2^{p+1} C^{p-1} \int_{\Omega_2} r^{-\frac{p-1}{2}}\abs{z}^{-\frac{p-1}{2}} \paren{\abs{v_k(r,z)}^2+\abs{v(r,z)}^2}r^3 d(r,z) \\ 
&\leq C_1 \paren{\norm{v_k}^2_{H^1_{\cyl}(r^3drdz)}+\norm{v}^2_{H^1_{\cyl}(r^3drdz)}}  R^{-(p-1)} \leq C_2 R^{-(p-1)}
\end{align*}
which is less or equal $\varepsilon$ if we choose $R>0$ large enough. 

\smallskip

Convergence on $\Omega_3$: Due to symmetry in $z$-direction it is enough to focus on $\tilde{\Omega}_3\coloneqq \{(r,z)\in \Omega: r<R, z\geq R\}$. Let $\alpha>0$ be arbitrary. Again by \eqref{abkling_5d} we obtain
\begin{align*}
\{(r,z)\in\tilde\Omega_3: v_k(r,z)>\alpha \} \subset \{(r,z)\in\tilde\Omega_3: r\  z^{\frac{1}{3}}\leq C_\alpha\} \eqqcolon S_\alpha,  
\end{align*}
where $C_\alpha=(C/\alpha)^{2/3}$ and $C$ is the constant from \eqref{abkling_5d}. The set $S_\alpha$ has finite measure since
\begin{align*}
\vert S_\alpha\vert \leq\int_R^\infty \int_0^{C_\alpha z^{-1/3}} r^3 dr\  dz= \frac{C_\alpha^4}{4} \int_R^\infty z^{-\frac{4}{3}} dz = \frac{3}{4}C_\alpha^4 R^{-\frac{1}{3}} <\infty.
\end{align*}

By the convergence principle from the remark above and since  by \eqref{H^1(r)gegenH^1(r^3)} $\Vert rv_k\Vert_{L^6(rdrdz)}\leq \|v_k\|_{H^1_{\cyl}(r^3drdz)}$ is bounded we obtain $\int_{S_\alpha} r^{p-1} \vert v_k-v\vert^{p+1} r^3d(r,z)\to 0 \text{ as } k\to\infty$ for $1 \leq p<5$.  
It remains to prove the convergence on $\tilde\Omega_3\setminus S_\alpha$. For allmost all $(r,z)\in \tilde\Omega_3\setminus S_\alpha$ we have that $v(r,z)= \lim_{k\to \infty} v_k(r,z)\leq \alpha$. Hence,
\begin{align*}
\int_{\tilde\Omega_3 \setminus S_\alpha} r^{p-1} \vert v_k-v\vert^{p+1} r^3d(r,z) \leq R^{p-1} (2\alpha)^{p-1}  \int_\Omega |v_k-v|^2 r^3d(r,z) \leq C \alpha^{p-1}.
\end{align*}
In summary, since $\alpha>0$ is arbitrary this shows \eqref{Inp+1cpt} on $\Omega_3$.

\smallskip

Convergence on $\Omega_4$: Again it is enough to focus on $\tilde{\Omega}_4\coloneqq \{(r,z)\in \Omega: r\geq R, 0\leq z< R\}$. Fix $z\in (0,R)$. Let us first show that
\begin{equation} \label{das_mal_zuerst}
\int_{\{r\geq R\}} r^{p-1}\vert v_k(r,z)-v(r,z)\vert^{p+1}r^3 dr \to 0 \text{ as } k\to\infty.
\end{equation} 
Since $v_k(r,\cdot)$ is nonincreasing in its last component we deduce
\begin{equation} \label{monotonie_trick}
\int_0^\infty r^q v_k^q(r,z) r\,dr \leq \frac{1}{z} \int_0^z \int_0^\infty r^q v_k^q(r,\zeta) r\, dr d\zeta \leq  \frac{1}{z} \int_\Omega r^q v_k^q(r,\zeta) r d(r,\zeta)\leq \frac{C}{z}
\end{equation}
for all $q\in [2,6]$ by \eqref{H^1(r)gegenH^1(r^3)}. Thus for $q\in [2,6]$ the sequence $\Vert \cdot v_k(\cdot,z)\Vert_{L^q((0,\infty),r dr)}$ is uniformly bounded in $k\in\N$.
Moreover, \eqref{abkling_5d} implies $v_k(r,z)\leq C(z) r^{-\frac{3}{2}}$ uniformly in $k\in\N$. Hence for $\tilde R >R$ 
\begin{align*}
\int_{\tilde{R}}^\infty  r^{p-1}\vert v_k(r,z)-v(r,z)\vert^{p+1}r^3 dr &\leq (2C(z))^{p-1}\int_{\tilde{R}}^\infty r^{-\frac{p-1}{2}} |v_k(r,z)-v(r,z)|^2 r^3dr \\
&\leq (2C(z))^{p-1}\tilde{R}^{\frac{1-p}{2}} \frac{C}{z} \mbox{ by } \eqref{monotonie_trick}. 
\end{align*}
The last term can be made arbitrarily small provided $\tilde{R}$ is chosen big enough. To finish the proof  of  \eqref{das_mal_zuerst} it remains to prove $\int_R^{\tilde{R}}  r^{p-1}\vert v_k(r,z)-v(r,z)\vert^{p+1}r^3 dr \to 0$ as $k\to\infty$. Since for almost all $z\in (0,R)$ we have $v_k(\cdot,z)\to v(\cdot,z)$ pointwise almost everywhere on $(R,\tilde{R})$ as well as the boundedness of $\Vert \cdot v_k(\cdot,z)\Vert_{L^6((0,\infty),r dr)}$ by \eqref{monotonie_trick} we can apply the convergence principle from the remark above and deduce
\begin{align*}
\int_R^{\tilde R} r^{p-1}\vert v_k(r,z)-v(r,z)\vert^{p+1}r^3 dr \to 0 \text{ as } k\to\infty.
\end{align*}
Hence \eqref{das_mal_zuerst} is accomplished for almost all $z\in (0,R)$.

\smallskip

Defining $\varphi_k(z)\coloneqq \int_{\{r\geq R\}} r^{p-1} \vert v_k(r,z)-v(r,z)\vert^{p+1} r^3dr$ we have $\varphi_k\to 0$ as $k\to\infty$ pointwise almost everywhere in $[0,R)$. The sequence $\paren{\varphi_k}_{k\in\N}$ is bounded in $L^1([0,R), dz)$ since by \eqref{hilfreiche_ungl}
\begin{align*}
\int_0^R\int_{\{r\geq R\}} r^{p-1} \vert v_k(r,z)-v(r,z)\vert^{p+1}  r^3 dr dz \leq C \int_\Omega r^{p-1}\paren{\vert v_k\vert^{p+1}+\vert v\vert^{p+1}} r^3d(r,z) \leq \tilde{C}.
\end{align*}
Moreover, for $p\in (1,3]$, the sequence $\paren{\varphi_k}_{k\in\N}$ is bounded in $W^{1,1}([0,R),dz)$ since
\begin{align*}
\norm{\frac{\partial \varphi_k}{\partial z}}_{L^1([0,R], dz)}^2 & \leq \paren{\int_0^R \int_R^\infty (p+1) r^{p-1} \vert v_k-v\vert^p \abs{\frac{\partial v_k}{\partial z}-\frac{\partial v}{\partial z}} r^3 dr dz}^2 \\
&\leq \paren{\int_\Omega (p+1) r^{p-1} \vert v_k-v\vert^p \abs{\frac{\partial v_k}{\partial z}-\frac{\partial v}{\partial z}} r^3 d(r,z)}^2 \\
& \leq C\int_{\Omega} r^{2p-2} \vert v_k-v\vert^{2p} r^3 d(r,z)\int_{\Omega} \abs{\frac{\partial v_k}{\partial z}-\frac{\partial v}{\partial z}}^2 r^3 d(r,z) \\
&= C \Vert r (v_k-v)\Vert_{L^{2p}(rdrdz)}^{2p} \int_{\Omega} \abs{\frac{\partial v_k}{\partial z}-\frac{\partial v}{\partial z}}^2 r^3 d(r,z) \leq C.
\end{align*}
Hence, by the compact embedding $W^{1,1}([0,R),dz)\hookrightarrow L^1([0,R),dz)$ we conclude that at least a subsequence of $(\varphi_k)_{k\in\N}$ is converging in $L^1 ([0,R), dz)$ to a limit function, which must be $0$ since we have already asserted the pointwise a.e. convergence to $0$ on $[0,R)$. This shows \eqref{Inp+1cpt} on $\Omega_4$ for $p\in (1,3]$. For $p\in (3,5)$ we make use of Hölder's interpolation, namely,
\begin{align*}
\norm{rv_k-rv}_{L^{p+1}_{\cyl}(\Omega_4, rdrdz)}^{p+1}\leq \norm{rv_k-rv}_{L^4_{\cyl}(\Omega_4, rdrdz)}^{4\theta} \norm{rv_k-rv}_{L^6_{\cyl}(\Omega_4, rdrdz)}^{6(1-\theta)} \leq \tilde{C} \norm{rv_k-rv}_{L^4_{\cyl}(\Omega_4, rdrdz)}^{4\theta} \to 0
\end{align*}
as $k\to\infty$, where $\theta\in (0,1)$ is chosen such that $p+1=4\theta+6(1-\theta)$, i.e., $\theta=\frac{5-p}{2}$.

\smallskip

The combination of convergences on $\Omega_1, \dots, \Omega_4$ finally proves \eqref{Inp+1cpt}.
\end{proof}

For our last lemma we need the notion of cylindrical $C_c^\infty$-functions which we introduce now.
\begin{Definition}
A function $u=u(r,z)$ belongs to $C_c^\infty([0,\infty)\times \R)$ if and only if $u\in C^\infty([0,\infty)\times \R)$, $\supp u$ is compact in $[0,\infty)\times\R$ and $\frac{\partial^j u}{\partial r^j}(0,z)=0$ for all odd integers $j\in 2\N-1$. 
\end{Definition}
\noindent 
{\bf Remark:} Since $u\in C_c^\infty([0,\infty)\times \R)$ is equivalent to $\tilde u\in C_c^\infty(\R^5)$ with $\tilde u(x) := u(|(x_1,\ldots,x_4)|,x_5)$ we see that $C_c^\infty([0,\infty)\times \R)$ is dense in $H^1_{\cyl}(r^3drdz)$.

\begin{Lemma}\label{rearrang_inequ}
For $u \in H^1_{\cyl}(r^3drdz)$ we have $\|u^\star\|\leq \|u\|$ where $\star$ denotes Steiner-symmetrization with respect to $z$ and $\|\cdot\|$ is the equivalent norm from Theorem~\ref{gsthm}. Moreover 
$$
I(u) \leq I(u^\star) \quad \mbox{ and } \quad I'(u)[u] \leq I'(u^\star)[u^\star].
$$
\end{Lemma}

\begin{proof} We begin by recalling several classical rearrangement inequalities from \cite{Lieb}, \cite{LiebLoss}. Recall first the Pólya-Szegö inequality 
\begin{equation} \label{polya-szegoe}
\int_{\R^n} \vert\nabla f^\circledast\vert^2 dx\leq \int_{\R^n}\vert\nabla f\vert^2 dx
\end{equation}
for $f\in H^1(\R^n)$ and $^\circledast$ denoting Schwarz-symmetrization (also called symmetrically decreasing rearrangement). Furthermore we have for $0\leq f,g\in L^2(\R^n)$ the classical rearrangement inequality
\begin{equation} \label{classical} 
\int_\R f g dx \leq \int_\R f^\circledast g^\circledast dx
\end{equation}
and the nonexpansivity of rearrangement
\begin{equation} \label{nonexpansivity}
\int_{\R^n} \vert f^\circledast -g^\circledast\vert^2 dx \leq \int_{\R^n} \abs{f-g}^2 dx.
\end{equation}
From \eqref{polya-szegoe} we immediately receive for $u\in H^1_{\cyl}(r^3drdz)$ that 
\begin{equation}\label{polya-szegoe_for_u}
\int_\R \vert\nabla_{z} u^\star\vert^2 dz\leq \int_\R \vert\nabla_{z}u\vert^2 dz.
\end{equation}
Next we want to establish a similar inequality for $\nabla_r u$. We do this first for $u\in C_c^\infty([0,\infty)\times \R)$. With the help of \eqref{nonexpansivity} we find that
\begin{align*}
\int_\R \abs{\frac{u^\star (r+t,z)-u^\star (r,z)}{t}}^2 dz \leq \int_\R \abs{\frac{u(r+t,z)-u(r,z)}{t}}^2 dz  
\end{align*}	
for almost all $r, t\in [0,\infty)$. Sending $t\to 0$ and using Fatou's lemma on the left side of the inequality yields 
\begin{align} \label{fuerccinf}
\int_\R \vert \nabla_r u^\star\vert^2 dz \leq \int_\R \vert \nabla_r u\vert^2 dz
\end{align}
for $u\in C_c^\infty([0,\infty)\times \R)$ and almost all $r\in [0,\infty)$. Since Steiner Symmetrization is continuous in $H^1$ (see Theorem 1 in \cite{Burchard}) we obtain by approximation that \eqref{fuerccinf} is indeed valid for all $u\in H^1_{cyl}(r^3drdz)$. Together with \eqref{polya-szegoe_for_u} we obtain $\int_\R \vert\nabla_{r,z} u^\star\vert^2 dz\leq \int_\R \vert\nabla_{r,z}u\vert^2 dz$ for almost all $r\geq 0$ and integration leads to 
\begin{align}\label{polsze2}
\int_\R\int_0^\infty \vert\nabla_{r,z} u^\star\vert^2 r^3 dr dz\leq \int_\R\int_0^\infty \vert\nabla_{r,z}u\vert^2 r^3 dr dz.
\end{align}
Fixing $r\in[0,\infty)$ and applying \eqref{classical} to $f(\cdot)=\esssup V-V(r,\cdot)$ and $g(\cdot)=u^2(r,\cdot)$ gives
\begin{align*} 
\int_\R \paren{\esssup V-V(r,\cdot)} u^2(r,\cdot) dz &\leq \int_\R \paren{\esssup V-V(r,\cdot)}^\star \big(u^2\big)^\star (r,\cdot) dz \\
&= \int_\R \paren{\esssup V-V(r,\cdot)} \paren{u^\star}^2(r,\cdot) dz.
\end{align*}
Using $\norm{u(r,\cdot)}_{L^2(\R)}=\Vert u^\star (r,\cdot)\Vert_{L^2(\R)}$ this results in
\begin{align} \label{Vwk}
\int_\R\int_0^\infty V(r,z) \paren{u^\star}^2 r^3 dr dz \leq \int_\R\int_0^\infty V(r,z) u^2 r^3 dr dz.
\end{align}
The combination of (\ref{polsze2}) and (\ref{Vwk}) yields the claimed inequality $\Vert u^\star\Vert^2 \leq \Vert u\Vert^2$.

\smallskip

Assumption (v) on $f$ allows to apply Theorem~5.1 in \cite{Brock} and to deduce
\begin{align} \label{Brock}
I'(u)[u] = \int_\Omega f(r,z,r^2u^2) u^2 r^3d(r,z) \leq \int_\Omega f(r,z,r^2u^{\star 2})u^{\star 2}r^3 d(r,z) = I'(u^\star)[u^\star].
\end{align}
Moroever, using (v) with $s=0$ shows that for all $r\in [0,\infty)$, $\sigma\geq 0$ the function $z \mapsto f(r,z,\sigma^2)$ is symmetrically nonincreasing in $z$ and hence 
$$
\Phi_\sigma(r,z,s) := F(r,z,r^2(s+\sigma)^2)- F(r,z,r^2s^2) = \int_{r^2s^2}^{r^2(s+\sigma)^2} f(r,z,t)\,dt
$$
is symmetrically nonincreasing in $z$. Applying once more Theorem~5.1 in \cite{Brock} yields 
$$
I(u) = \int_\Omega \frac{1}{2r^2}F(r,z,r^2u^2) r^3d(r,z) \leq \int_\Omega \frac{1}{2r^2}F(r,z,r^2u^{\star 2}) r^3 d(r,z) = I(u^\star).
$$
This finishes the proof of the lemma.
\end{proof}

\section{Proof of Theorem \ref{gsthm}} \label{3}

\begin{proof}
Recall from Lemma~\ref{weak_continuity} the definition $I(u)\coloneqq \int_\Omega \frac{1}{2r^2} F(r,z,r^2u^2) r^3 d(r,z)$ for $u\in H^1_{\cyl}(r^3 drdz)$. We show that the assumptions (i)-(iii) of Theorem 12 in \cite{SW} are satisfied. Let $\varepsilon>0$. The growth assumptions (i) and (ii) on $f$ imply that for every $\epsilon>0$ there exists $C_\epsilon>0$ such that the global estimate $0 \leq f(r,z,s) \leq \epsilon + C_\epsilon |s|^\frac{p-1}{2}$ holds. Together with \eqref{hilfreiche_ungl} we obtain 
\begin{align*}
\vert I'(u)[v]\vert &= \abs{ \int_\Omega f(r,z,r^2u^2)uv r^3d(r,z) }\\
& \le \varepsilon \int_\Omega \vert ru\vert \vert rv\vert rd(r,z) + C_\epsilon\int_\Omega \vert ru\vert^p \vert rv\vert rd(r,z) \\
&\le \varepsilon C \norm{u}_{H^1_{cyl}(r^3drdz)}\norm{v}_{H^1_{cyl}(r^3drdz)} + \tilde C_\epsilon \norm{u}^p_{H^1_{cyl}(r^3drdz)} \norm{v}_{H^1_{cyl}(r^3drdz)} 
\end{align*}
Taking the supremum over all $v\in H^1_{cyl}(r^3drdz)$ with $\|v\|_{H^1_{cyl}(r^3drdz)}=1$ we see that
\begin{equation} \label{ws1}
I'(u)= o(\norm{u}) \mbox{ as } u\to 0.
\end{equation}
Moreover, due to assumption (iii) on $f$ the map
\begin{equation} \label{ws2} 
s\mapsto \frac{I'(su)[u]}{s} = \int_\Omega f(r,z,s^2 r^2u^2) u^2 r^3 d(r,z) \mbox{ is strictly increasing for all $u\neq 0$ and $s>0$.}
\end{equation}
Next we claim that 
\begin{equation} \label{ws3}
\frac{I(su)}{s^2}\to \infty \mbox{ as } s\to\infty \mbox{ uniformly for $u$ on weakly compact subsets $W$ of $H^1_{cyl}(r^3drdz)\setminus \{0\}$.} 
\end{equation}
Suppose not. Then there are $(u_k)_{k\in\N}\subset W$ and $s_k\to\infty$ as $k\to \infty$ such that $\frac{I(s_k u_k)}{s_k^2}$ is bounded as $k\to\infty$. But along a subsequence we have $u_k \rightharpoonup u \neq 0$ and $u_k(x)\to u(x)$ pointwise almost everywhere. Let $\Omega^\sharp := \{(r,z)\in \Omega: u(r,z)\not =0\}$. Then $|\Omega^\sharp|>0$ and on $\Omega^\sharp$ we have $\vert s_k u_k(r,z)\vert \to \infty$ as $k\to \infty$.  Fatou's lemma and assumption (iv) on $F$ imply
\begin{align*} 
\frac{I(s_k u_k)}{s_k^2} = \int_\Omega \frac{F(r,z,s_k^2r^2u_k^2)}{2s_k^2 r^2} r^3 d(r,z) \geq  \int_{\Omega^\sharp} \frac{F(r,z,s_k^2 r^2 u_k^2)}{2 s_k^2 r^2 u_k^2} u_k^2 r^3 d(r,z) \to \infty \text{ as } k\to\infty,
\end{align*}
a contradiction. In summary, \eqref{ws1}, \eqref{ws2}, \eqref{ws3} imply that (i)-(iii) of Theorem 12 in \cite{SW} are satisfied. 

\smallskip

Now we take a sequence $(u_k)_{k\in\N}\subset M$ such that $J(u_k)\to \inf_{M} J$ as $k\to\infty$. Since $\Vert\nabla_{r,z} \abs{u_k}\Vert_{L^2}= \Vert\nabla_{r,z} u_k\Vert_{L^2}$ we can assume that $u_k \geq 0$ for all $k\in\N$. Then Theorem 12 in \cite{SW} guarantees that for every $k$ there is a unique $t_k>0$ such that $v_k:= t_k u_k^\star\in M$. We show next that $t_k\leq 1$ for all $k\in\N$. Assume $t_k>1$. Then
\begin{align*}
\int_\Omega f(r,z,r^2 u_k^{\star 2})u_k^{\star 2} r^3d(r,z) &< \int_\Omega f(r,z,t_k^2 r^2 u_k^{\star 2})u_k^{\star 2} r^3d(r,z) \quad \mbox{ by assumption (iii)} \\
&= \Vert u_k^\star \Vert^2  \quad \mbox{ since } t_k u_k^\star\in M \\
& \leq \Vert u_k\Vert ^2 \quad \mbox{ by Lemma~\ref{rearrang_inequ}} \\
& = \int_\Omega f(r,z,r^2u_k^2)u_k^2 r^3 d(r,z) \quad \mbox{ since } u_k\in M.
\end{align*}
This contradicts the inequality $I'(u_k)[u_k]\leq I'(u_k^\star)[u_k^\star]$ from Lemma~\ref{rearrang_inequ} and thus $t_k\leq 1$ for all $k\in\N$.  

\smallskip

Next notice that for fixed $(r,z,s)\in [0,\infty)\times\R\times [0,\infty)$ and $t\in (0,1]$ one has
\begin{align*}
\frac{d}{dt} \paren{t^2f(r,z,s^2)s^2-F(r,z,t^2s^2)} = 2ts^2\paren{f(r,z,s^2)-f(r,z,t^2s^2)} > 0
\end{align*}
since $f$ is strictly increasing in its last variable by assumption (iii). This shows that the map $t\mapsto t^2 f(r,z,s^2)s^2-F(r,z,t^2s^2)$ is strictly increasing for $t\in [0,1]$. From this monotonicity and the inequality $I(t_ku_k)\leq I(t_k u_k^\star)$ from Lemma~\ref{rearrang_inequ} we conclude
\begin{align}
2J(v_k)&=\int_\Omega \paren{t_k^2 \vert\nabla_{r,z} u_k^\star\vert^2+V(r,z)t_k^2u_k^{\star 2} -  \frac{1}{r^2}F(r,z,r^2 t_k^2u_k^{\star 2})}r^3d(r,z) \nonumber \\
&\le \int_\Omega \paren{t_k^2 \vert\nabla_{r,z} u_k\vert^2+V(r,z)t_k^2 u_k^2- \frac{1}{r^2}F(r,z,r^2 t_k^2u_k^2)}r^3d(r,z) \nonumber \\
&=\int_\Omega \frac{1}{r^2} \paren{ f(r,z,r^2u_k^2) t_k^2 r^2 u_k^2 -F(r,z,r^2 t_k^2 u_k^2) }r^3d(r,z) \label{min_seq}\\
&\leq \int_\Omega \frac{1}{r^2} \paren{f(r,z,r^2u_k^2)r^2 u_k^2-F(r,z,r^2 u_k^2)} r^3d(r,z) \nonumber \\
& =2J(u_k). \nonumber
\end{align}
So $(v_k)_{k\in\N}\subset M$ is also a minimizing sequence for $J$ which belongs to $K_{4,1}$. The boundedness of $(v_k)_{k\in \N}$ is established in Proposition~14 in \cite{SW}. Hence, we find $v_\infty\in H_{\cyl}^1(r^3drdz)$ such that $v_k\rightharpoonup v_\infty$ in $H_{\cyl}^1(r^3drdz)$ along a subsequence as $k\to\infty$. In addition, $v_\infty\in K_{4,1}$ due to Lemma~\ref{Kabg} and $v_\infty\not =0$ by Proposition~14 in \cite{SW} where instead of the weak sequential continuity of $I$ on all of $H^1_{\cyl}(r^3drdz)$ we use it only on $K_{4,1}$ as stated in Lemma~\ref{weak_continuity}.

\smallskip

Let us show that $v_\infty\in M$. Since $v_\infty\not =0$ we can choose $t_\infty > 0$ such that $t_\infty v_\infty\in M$. In the same manner as before for the sequence $t_k$ we can show that $t_\infty\leq 1$. Assume $t_\infty<1$. Then as in \eqref{min_seq} and using the weak sequential continuity on $K_{4,1}$ as shown in Lemma~\ref{weak_continuity} we find 
\begin{align*}
2 J(t_\infty v_\infty)& < \int_\Omega \frac{1}{r^2} \paren{f(r,z,r^2v_\infty^2)r^2 v_\infty^2 - F(r,z,r^2v_\infty^2)} r^3 d(r,z) \\
&=\lim_{k\to\infty} \int_\Omega \frac{1}{r^2}\paren{f(r,z,r^2v_k^2)r^2v_k^2-F(r,z,r^2v_k^2)} r^3 d(r,z) \\
&= 2 \inf_M J \leq 2 J(t_\infty v_\infty)
\end{align*}
which is a contradiction. So $t_\infty=1$ and thus $v_\infty\in M$. Then by the weak lower semi-continuity of $\norm{\cdot}$ and once again the weak sequential continuity of $I$ we conclude
\begin{align*}
J(v_\infty) \leq \liminf_{k\to\infty} J(v_k) = \inf_M J \leq J(v_\infty).
\end{align*}
Hence, $v_\infty\in K_{4,1}$ is a minimizer of $J$ on $M$, i.e., a ground state of (\ref{scaequ}) which is Steiner symmetric in $z$ with respect to $\{z=0\}$. 
\end{proof}

\section*{Appendix}
Here we prove that the condition $V\geq 0$ and $\inf_{B_R^c} V>0$ for some $R>0$ implies that on $H^1_{\cyl}(r^3drdz)$ the expression  
$\left(\int_\Omega \paren{\vert\nabla_{r, z}u\vert^2+ V(r,z)u^2} r^3 d(r,z)\right)^\frac{1}{2}$ is an equivalent norm. Suppose not. Then there is a sequence $(u_k)_{k\in\N}$ such that $\Vert u_k\Vert_{L^2(r^3drdz)}=1$ and $\int_\Omega \paren{\vert \nabla_{r,z} u_k\vert^2+V(r,z)u_k^2}r^3d(r,z)\to 0$ as $k\to\infty$. In particular, 
\begin{align} \label{aequivnormbew}
\int_\Omega \vert\nabla_{r,z}u_k\vert^2 r^3d(r,z)\to 0 \text{ and } \int_{B_R^c} u_k^2 r^3d(r,z) \to 0 \text{ as } k\to \infty.
\end{align}
Let $\chi$ denote a smooth cut-off function such that $\chi(r,z)=1$ for $0\leq \sqrt{r^2+z^2}<R$ and $\chi(r,z)=0$ for $\sqrt{r^2+z^2}\geq R+1$. Then $v_k\coloneqq \chi u_k\in H^1_{0,\cyl}(B_{R+1},r^3drdz)$ and
\begin{align*}
\vert\nabla_{r,z}v_k \vert^2=\chi^2 \vert\nabla_{r,z}u_k \vert^2+\vert\nabla_{r,z}\chi \vert^2u_k^2+2u_k\chi \nabla_{r,z}u_k\cdot\nabla_{r,z}\chi.
\end{align*}
Hence, by (\ref{aequivnormbew})
\begin{align}
\int_\Omega \vert\nabla_{r,z} v_k\vert^2 r^3d(r,z)&\leq 2\int_\Omega \chi^2\vert\nabla_{r,z}u_k\vert^2r^3d(r,z) +2\int_\Omega u_k^2 \vert\nabla_{r,z} \chi\vert^2 r^3d(r,z) \label{abschaetz_nabla_vk} \\
&\leq 2\int_\Omega\vert\nabla_{r,z} u_k\vert^2 r^3d(r,z)+2\|\nabla_{r,z} \chi\|^2_\infty \int_{B_{R+1}\setminus B_R} u_k^2 r^3d(r,z) \to 0 \text{ as } k\to\infty. \nonumber
\end{align}
In particular, $\int_{B_{R+1}} \vert\nabla_{r,z} v_k\vert^2r^3d(r,z)\to 0$ as $k\to\infty$. By Poincaré's inequality, $\Vert u_k\Vert_{L^2(r^3drdz)}=1$ and \eqref{aequivnormbew} we see
\begin{align*}
C_P \int_{B_{R+1}} \vert\nabla_{r,z} v_k\vert^2 r^3d(r,z) \geq \int_{B_{R+1}} v_k^2r^3d(r,z)\geq \int_{B_R} u_k^2 r^3d(r,z) = 1-o(1),
\end{align*}
contradicting \eqref{abschaetz_nabla_vk}. \qed

\section*{Acknowledgement}
We gratefully acknowledge financial support by the Deutsche Forschungsgemeinschaft (DFG) through CRC 1173.


\begin{thebibliography}{la Bre}

\bibitem{Fortunato} Azzollini, A., Benci, V., D’Aprile, T. and Fortunato, D.: Existence of static solutions of the semilinear Maxwell equations. Ricerche di matematica, 55(2), 123-137, (2006).

\bibitem{BDPR} Bartsch, T., Dohnal, T., Plum, M. and Reichel, W.: Ground states of a nonlinear curl-curl problem in cylindrically symmetric media. arXiv preprint arXiv:1411.7153, (2014).

\bibitem{bartsch_mederski} Bartsch, T. and Mederski, J.: Ground and bound state solutions of semilinear
time-harmonic Maxwell equations in a bounded domains. Arch. Ration. Mech. Anal., 215(1), 283--306, (2015).

\bibitem{bartsch_mederski_2} Bartsch, T. and Mederski, J.: Nonlinear time-harmonic Maxwell equations in an anisotropic bounded medium. arXiv preprint arXiv:1509.01994, (2015).

\bibitem{benci_fortunato_archive} Benci, V. and Fortunato, D.: Towards a unified field theory for classical electrodynamics. Arch. Ration. Mech. Anal., 173(3), 379--414, (2004). 

\bibitem{Brock} Brock, F.: Continuous rearrangement and symmetry of solutions of elliptic problems. Proceedings of the Indian Academy of Sciences-Mathematical Sciences. Vol. 110. No. 2. Springer India, (2000).

\bibitem{Burchard} Burchard, A.: Steiner symmetrization is continuous in $W^{1,p}$. Geometric \& Functional Analysis GAFA 7.5: 823-860, (1997).

\bibitem{daprile_siciliano} D'Aprile, T. and Siciliano, G.: Magnetostatic solutions for a semilinear perturbation of the
              {M}axwell equations. Adv. Differential Equations, 16(5-6), 435--466, (2011).

\bibitem{Lieb} Lieb, E. H.: Existence and uniqueness of the minimizing solution of Choquard's nonlinear equation. Studies in Applied Mathematics 57: 93-105, (1977).   

\bibitem{LiebLoss} Lieb, E. H. and Loss, M.: Analysis, volume 14 of graduate studies in mathematics. American Mathematical Society, Providence, RI, 4, (2001).

\bibitem{Lions2} Lions, P.-L.: Minimization problems in $L^1(\R^3)$. Journal of Functional Analysis, 41(2), 236-275, (1981).

\bibitem{Lions} Lions, P.-L.: Symétrie et compacité dans les espaces de Sobolev. Journal of Functional Analysis 49.3: 315-334, (1982).  

\bibitem{Liu} Liu, S.: On superlinear problems without the Ambrosetti and Rabinowitz condition. Nonlinear Analysis: Theory, Methods \& Applications, 73(3), 788-795, (2010).

\bibitem{Mederski_2014} Mederski, J.: Ground states of time-harmonic semilinear Maxwell equations in $\mathbb{R}^3$ with vanishing permittivity. arXiv preprint arXiv:1406.4535, (2014), to appear in Arch. Ration. Mech. Anal.

\bibitem{Nehari1} Nehari, Z.: On a class of nonlinear second-order differential equations. Transactions of the American Mathematical Society 95.1: 101-123, (1960).

\bibitem{Nehari2} Nehari, Z.: Characteristic values associated with a class of nonlinear second-order differential equations. Acta Mathematica 105.3: 141-175, (1961).

\bibitem{Stuart} Stuart, C. A.: A variational approach to bifurcation in $L^p$ on an unbounded symmetrical domain. Mathematische Annalen, 263(1), 51-59, (1983). 

\bibitem{SW} Szulkin, A. and Weth, T.: The method of Nehari manifold. Handbook of nonconvex analysis and applications: 597-632, (2010).

\bibitem{Willem} Willem, M.: Minimax theorems. Vol. 24. Springer Science \& Business Media, (1997).

\end{thebibliography}
\end{document}